\documentclass[aap,reqno,preprint]{imsart}

\RequirePackage{amsthm,amsmath,amsfonts,amssymb,mathtools}
\RequirePackage[numbers,sort&compress]{natbib}
\RequirePackage[hidelinks]{hyperref}

\startlocaldefs
\usepackage{microtype}
\usepackage{enumitem}
\newtheorem{theorem}{\textsc{Theorem}}[section]
\newtheorem{proposition}[theorem]{\textsc{Proposition}}
\newtheorem{lemma}[theorem]{\textsc{Lemma}}
\numberwithin{equation}{section}
\renewcommand{\proofname}{\textsc{Proof}}

\newcommand{\E}{\mathbb E}
\newcommand{\R}{\mathbb R}
\newcommand{\Cov}{\operatorname{Cov}}
\newcommand{\dd}{\mathop{}\!\mathrm d}
\DeclareMathOperator*{\argmax}{arg\,max}
\DeclareMathOperator{\id}{id_\R}
\DeclareMathOperator{\Var}{Var}
\endlocaldefs

\begin{document}
\begin{frontmatter}

	\title{The one-period Kyle model has one
		equilibrium}

	\runtitle{Uniqueness in the Gaussian Kyle Model}

	\begin{aug}
		\author[B]{\fnms{Paulo~K.}~\snm{Monteiro}%
			\ead[label=e2]{paulo.klinger@fgv.br}}
		\author[C]{\fnms{Rabee}~\snm{Tourky}%
			\ead[label=e3]{rabee.tourky@anu.edu.au}}

		\address[B]{FGV EPGE Brazilian School of Economics and Finance,
			Rio de Janeiro, Brazil%
			\printead[presep={,\ }]{e2}}
		\address[C]{Research School of Economics,
			Australian National University,
			Canberra ACT 2601, Australia%
			\printead[presep={,\ }]{e3}}
	\end{aug}

	\begin{abstract}
		Let $V$ and $U$ be independent standard normal random variables. For
		each Borel-measurable function $\phi\colon\R\to\R$, let
		$P_\phi\colon\R\to\R$ be a Borel version of the inverse regression
		$y\mapsto\E[V\mid\phi(V)+U=y]$. We prove the rigidity result
		\[
			\phi(v)\in\argmax_{x\in\R}
			\E\!\left[(v-P_\phi(x+U))x\right]
			\qquad\text{for every }v\in\R
		\]
		if and only if $\phi=\id$, the identity function.

		The rigidity result implies that the
		one-period Gaussian Kyle~(1985) insider-trading model has a unique
		Borel-measurable equilibrium strategy, namely Kyle's affine strategy.
		Building on preliminary results of \citeauthor*{MMT17}~(2017), the
		proof establishes that equilibrium attains a sharp
		universal upper bound on this noise traders' loss and that any strategy
		attaining the bound must be affine almost everywhere.
	\end{abstract}




\end{frontmatter}

\setcounter{tocdepth}{2}
\tableofcontents
\clearpage

\section{Introduction}
\label{sec:intro}

The one-period Gaussian model in \citet{Kyle85} is a noisy signalling
model of insider trading. An informed trader observes the asset value
\[
	V\sim\mathcal N(\mu,\tau^2),\qquad \tau>0\,,
\]
and submits demand $\phi(V)$, where
$\phi\colon\R\to\R$ is a Borel-measurable \emph{strategy}. Market
makers observe total demand
\[
	Y_\phi=\phi(V)+U,\qquad
	U\perp V,\qquad
	U\sim\mathcal N(0,\sigma^2),\quad \sigma>0\,,
\]
where $U$ is noise traders' demand, and set the price $P(Y_\phi)$,
where $P\colon\R\to\R$ is a Borel-measurable \emph{pricing rule}.

A pair $(\phi,P)$ is an \emph{equilibrium} if the following two
conditions hold:
\begin{enumerate}
	\item \emph{Profit maximisation.}\label{item:profit-max}
	      For every $v\in\R$,
	      \[
		      \phi(v)\in
		      \argmax_{x\in\R}
		      \E\!\left[(v-P(x+U))x\right]\,.
	      \]

	\item \emph{Market efficiency.}
	      \[
		      P(Y_\phi)=\E[V\mid Y_\phi]
		      \qquad\text{almost surely}\,.
	      \]
\end{enumerate}
The first condition is expected-profit maximisation by a risk-neutral
informed trader; the second is competitive risk-neutral pricing by the
inverse regression.

If an equilibrium strategy $\phi$ is affine, then $(V,Y_\phi)$ is
jointly Gaussian. The equilibrium conditions and Gaussian regression
yield \emph{Kyle's affine equilibrium}:
\begin{align}
	\phi(v)
	 & =\frac{\sigma}{\tau}(v-\mu)
	 &                                 & \text{for every }v\in\R\,,
	\label{eq:kyle-strategy}                                        \\
	P(Y_\phi)
	 & =\mu+\frac{\tau}{2\sigma}Y_\phi
	 &                                 & \text{almost surely}\,.
	\label{eq:kyle-price}
\end{align}
This calculation does not rule out non-affine equilibria. We prove
that Kyle's affine equilibrium is unique:\footnote{Equilibrium pricing
	rules are identified as equivalence classes modulo
	Lebesgue-almost-everywhere equality. If profit maximisation is imposed
	only for Lebesgue-almost every $v$, the equilibrium strategy is
	likewise unique modulo Lebesgue-null sets; see
	Appendix~\ref{subsec:ae-maximisation}.}
\begin{theorem}
	\label{thm:kyle}
	A pair $(\phi,P)$ is an equilibrium if and only if
	\eqref{eq:kyle-strategy} and \eqref{eq:kyle-price} hold.
\end{theorem}

Empirical analogues of the coefficient
$\lambda=\tau/(2\sigma)$ in \eqref{eq:kyle-price}, known as
\emph{Kyle's lambda}, are widely used to measure price impact and
illiquidity \citep{Hasbrouck91,BS96,Amihud02}. Related measures are
used to study adverse selection and whether prices reveal informed
trading \citep{GH88,CDF15}. Their structural interpretation in the
Gaussian Kyle model requires a unique equilibrium prediction. If a
non-affine equilibrium existed, it would induce a non-jointly-Gaussian
distribution of the observables
$\bigl(P(Y_\phi),Y_\phi\bigr)$, unlike Kyle's affine equilibrium.
Hence, the distribution of observables would depend on equilibrium
selection, potentially rendering the model incomplete or partially
identified \citep{Tamer03,CilibertoTamer09}.
Theorem~\ref{thm:kyle} rules out this ambiguity: the model's
prediction for Kyle's lambda is independent of equilibrium selection.

Related uniqueness results concern modified models. \citet{RV94}
allow general value distributions but let the insider observe noise
demand before choosing demand. In continuous time, \citet{Back92}
establishes uniqueness of the pricing rule within a specified class;
see \citet{Cetin18}. With uniformly distributed asset value and noise
demand and a broad class of trading penalties, \citet*{CCG22} prove
uniqueness, up to their notion of indistinguishability, within a
specified class of nondecreasing strategies. By contrast,
\citet{CE00} exhibit a unique single-auction equilibrium with a
strongly nonlinear pricing rule for a Bernoulli value distribution.
The conclusion of Theorem~\ref{thm:kyle} depends
substantively on the Gaussian specification.

For the frictionless one-period Gaussian model,
\citet*{BKL05,BKL13} introduce a complex-analytic approach to
uniqueness, which \citet*{MMT17} develop further. They give an example
showing that the first-order conditions alone do not suffice to
establish uniqueness, and prove uniqueness when the equilibrium
strategy agrees locally with a function admitting a unique analytic
continuation. To our knowledge, uniqueness remained open for general
Borel-measurable strategies. Our theorem resolves the question without
modifying the model or imposing regularity beyond Borel measurability
of the strategy and pricing rule.

Our proof has three stages, each combining an economic insight with a
brief probabilistic argument.
\begin{enumerate}[label=\roman*.]
	\item \emph{Market-efficiency bounds.}
	      Market efficiency imposes sharp universal upper and lower
	      bounds on the noise traders' total expected loss over all
	      Borel-measurable insider strategies. Either bound is attained
	      if and only if the strategy agrees Lebesgue-almost everywhere
	      with the corresponding affine extremiser.

	\item \emph{Accounting identity for monotone strategies.}
	      Any nondecreasing strategy with nonnegative conditional
	      expected profit has finite ex ante expected profit equal to
	      the noise traders' total expected loss.

	\item \emph{Saturation at equilibrium.}
	      Pointwise profit maximisation makes the equilibrium strategy
	      strictly increasing and conditional expected profit
	      nonnegative, so the accounting identity applies. Their common
	      value attains the sharp upper bound from (i).
\end{enumerate}
Monotonicity, pointwise profit maximisation, and Gaussian regression
then yield uniqueness. The proof uses conditional resampling, a
classical Gaussian variance bound, supermodularity, convex analysis,
and lattice arguments.

The remainder of the paper has the following organization. The
next section records seven auxiliary lemmas, with their
provenance noted before the relevant statements. The reader may
proceed directly to Section~\ref{sec:quantities}, skipping these
preliminaries. Section~\ref{sec:quantities} defines noise-trader loss
and informed trader profit and records their basic properties.
Sections~\ref{sec:covariance-bound}--\ref{sec:covariance-saturation}
then establish, respectively, the sharp universal bounds and their
equality cases, the accounting identity, and saturation at
equilibrium. Section~\ref{sec:rigidity} proves the rigidity theorem
and Theorem~\ref{thm:kyle}.

\section{Preliminaries}

\label{sec:preliminaries}

This section fixes notation and records seven auxiliary lemmas used
below. Appendix~\ref{app:proofs-preliminaries} gives proofs for
completeness.

Let $U$ and $V$ be independent standard normal random variables. Fix a
Borel-measurable function $\phi\colon\R\to\R$ and set
\[
	Y=\phi(V)+U\,.
\]
We assume no further regularity or moment condition on $\phi$; in
particular, $\phi(V)$ need not be integrable. Throughout, we refer to a version of
$y\mapsto\E[V\mid Y=y]$ as the \emph{inverse regression}, and to the
random variable $\E[V\mid Y]$ as the \emph{posterior mean}. We denote by
$\id\colon\R\to\R$ the identity map.

\subsection{Densities of
	\texorpdfstring{$(V,Y)$ and $Y$}{(V,Y) and Y}}
\label{subsec:densities}
Let
\[
	\varrho(t):=\frac{1}{\sqrt{2\pi}}e^{-t^2/2}\qquad\text{for every }t\in\R\,,
\]
denote the standard Gaussian density.

\begin{lemma}
	\label{lem:densities}
	The joint law of $(V,Y)$ has density
	\[
		(v,y)\longmapsto\varrho(v)\varrho\bigl(y-\phi(v)\bigr)\,,
	\]
	and the law of $Y$ has the everywhere positive density
	\[
		p_Y(y):=\int_\R
		\varrho(v)\varrho\bigl(y-\phi(v)\bigr)\,\dd v
		\qquad\text{for every }y\in\R\,.
	\]
\end{lemma}

\subsection{Continuous version of the inverse regression}
\label{subsec:continuous-version}

The inverse-regression and Gaussian-smoothing results in this subsection
and the next were established in \citep{MMT17}, where substantially
stronger complex-analytic regularity is proved.

Define $P\colon\R\to\R$ by
\[
	P(y):=\frac{1}{p_Y(y)}\int_\R
	v\varrho(v)\varrho\bigl(y-\phi(v)\bigr)\,\dd v,
	\qquad\text{for every }y\in\R\,.
\]
The integral is absolutely convergent because
$\varrho\leq(2\pi)^{-1/2}$ and $V\in L^1$.
Lemma~\ref{lem:densities} gives $p_Y(y)\in(0,\infty)$. Thus, $P$ is
well defined.

\begin{lemma}
	\label{lem:P-properties}
	The function $P$ is continuous,
	\[
		P(Y)=\E[V\mid Y]\qquad\text{almost surely}\,,
	\]
	and $P(Y)\in L^2$.
\end{lemma}

For each $y\in\R$, define the probability measure
\[
	\pi_y(\dd v):=
	\frac{\varrho(v)\varrho\bigl(y-\phi(v)\bigr)}{p_Y(y)}\,\dd v\,.
\]
By Lemma~\ref{lem:densities} and Bayes' formula, $y\mapsto\pi_y$ is a
version of the conditional law of $V$ given $Y=y$. Moreover,
\[
	P(y)=\int_\R v\,\pi_y(\dd v),\qquad\text{for every }y\in\R\,.
\]

\subsection{Gaussian smoothing of the inverse regression}
\label{subsec:smoothing}

The following lemma is the principal regularity input used in the
remainder of the paper.

Define the Gaussian smoothing of the continuous version by
\begin{equation}
	F(x):=\E[P(x+U)],\qquad\text{for every }x\in\R\,.
	\label{eq:F-definition}
\end{equation}

\begin{lemma}
	\label{lem:F-properties}
	The expectation in \eqref{eq:F-definition} is finite for every
	$x\in\R$, and $F\colon\R\to\R$ is continuously differentiable. For
	every $x\in\R$,
	\begin{equation}
		F(x)=\int_\R P(y)\varrho(y-x)\,\dd y\,,
		\label{eq:F-convolution}
	\end{equation}
	and
	\begin{equation}
		\begin{aligned}
			F'(x)
			 & =\int_\R P(y)(y-x)\varrho(y-x)\,\dd y \\
			 & =\int_\R uP(x+u)\varrho(u)\,\dd u\,.
		\end{aligned}
		\label{eq:F-derivative}
	\end{equation}
	The integrals in \eqref{eq:F-convolution} and
	\eqref{eq:F-derivative} are absolutely convergent. For every Borel
	version $\widetilde P$ of $\E[V\mid Y]$ and every $x\in\R$,
	\[
		\widetilde P(x+U)\in L^1,\qquad
		\E[\widetilde P(x+U)]=F(x)\,.
	\]
	Moreover,
	\[
		F(\phi(V))
		=\E[\E[V\mid Y]\mid V]
		\qquad\text{almost surely}\,.
	\]
\end{lemma}

\subsection{Consequences of pointwise maximisation}
\label{subsec:first-order-conditions}

The next lemma supplies the monotonicity of $\phi$ used in
Proposition~\ref{prop:balance},
together with the first-order and
convex-analytic identities used in
Proposition~\ref{prop:covariance-saturation}.
Its first-order and monotonicity
conclusions were established in \citep{MMT17}. We formulate them in terms of the value function $W$. The
convex-analytic facts used in the proof are standard and may be found
in \citet{Rockafellar70}.

Define
\begin{equation*}
	W(v):=\sup_{x\in\R}x\bigl(v-F(x)\bigr),\qquad\text{for every }v\in\R\,,
\end{equation*}
where the value $+\infty$ is allowed. The function $W$ is nonnegative
because $x=0$ is admissible, and convex because it is a pointwise
supremum of affine functions of $v$.

\begin{lemma}
	\label{lem:foc}
	If
	\[
		\phi(v)\in\argmax_{x\in\R}x\bigl(v-F(x)\bigr)
		\qquad\text{for every }v\in\R\,,
	\]
	then $W$ is real-valued and, for every $v\in\R$,
	\begin{align}
		F(\phi(v))+\phi(v)F'(\phi(v)) & =v
		\label{eq:foc}                                           \\
		W(v)                          & =\phi(v)^2F'(\phi(v))\,.
		\label{eq:W-identity}
	\end{align}
	Moreover, $W$ is locally absolutely continuous and
	\begin{equation}
		W'(v)=\phi(v)
		\qquad\text{for Lebesgue-almost every }v\in\R\,.
		\label{eq:W-prime}
	\end{equation}
	The map $\phi$ is strictly increasing.
\end{lemma}

\subsection{Monotonicity of the inverse regression and its smoothing}
\label{subsec:monotonicity}

For a probability measure $\pi$ on $\R$ and $f\in L^1(\pi)$, write
\[
	\E_\pi[f]:=\int_\R f\,\dd\pi\,.
\]
For $f,g\in L^2(\pi)$, write
\[
	\Cov_\pi(f,g):=\E_\pi[fg]-\E_\pi[f]\,\E_\pi[g]\,.
\]

Under the hypothesis of Lemma~\ref{lem:foc}, the strategy $\phi$ is
strictly increasing. The next lemma shows that, under the same
hypothesis, the continuous inverse regression $P$ and its Gaussian
smoothing $F$ are also strictly increasing; in particular, it supplies
the strict positivity of $F'$ used in
Proposition~\ref{prop:covariance-saturation}. Its proof uses the
score--covariance identity for exponential families \citep{Brown86}.

\begin{lemma}
	\label{lem:posterior}
	Under the hypothesis of Lemma~\ref{lem:foc}, for every $y\in\R$, the
	functions $\id$ and $\phi$ lie in $L^2(\pi_y)$. The function $P$ is
	continuously differentiable, with
	\begin{equation}
		P'(y)=\Cov_{\pi_y}(\id,\phi)>0,\qquad\text{for every }y\in\R\,.
		\label{eq:P-prime-cov}
	\end{equation}
	Consequently,
	\begin{equation}
		F'(x)>0,\qquad\text{for every }x\in\R\,.
		\label{eq:F-prime-positive}
	\end{equation}
	Thus, $P$ and $F$ are strictly increasing.
\end{lemma}

\subsection{Conditional resampling}
\label{subsec:resampling}

The following lemma records the conditional-resampling and
exchangeable-pair identities. It is a textbook application of the transfer theorem in
\citet[Theorem~8.17]{Kallenberg21}. The construction yields an
exchangeable pair in the sense
of \citet{Chatterjee07}.

\begin{lemma}
	\label{lem:resampling}
	Set $V^{(0)}:=V$. On an extension of the underlying probability
	space, there exists a random variable $V^{(1)}$ such that $V^{(0)}$
	and $V^{(1)}$ are conditionally independent and identically
	distributed given $Y$. Moreover, the following hold:
	\begin{enumerate}[label=(\roman*)]
		\item\label{item:resampling-exchangeable}
		      the triples $(V^{(0)},V^{(1)},Y)$ and
		      $(V^{(1)},V^{(0)},Y)$ have the same law;

		\item\label{item:resampling-marginal}
		      the pairs $(V^{(1)},Y)$ and $(V,Y)$ have the same law;

		\item\label{item:resampling-posterior-mean}
		      $\E[V^{(1)}\mid Y,V^{(0)}]=P(Y)$ almost surely;

		\item\label{item:resampling-smoothed-mean}
		      $\E[V^{(1)}\mid V^{(0)}]
			      =F\bigl(\phi(V^{(0)})\bigr)$ almost surely.
	\end{enumerate}
\end{lemma}

\subsection{A lower Gaussian variance bound}
\label{subsec:lower-gaussian-variance}

The following classical inequality is the bounded Lipschitz case of
the lower Gaussian variance bound of \citet{Cacoullos82}. This is the
only form needed in the proof of
Proposition~\ref{prop:covariance-saturation}.

\begin{lemma}[Lower Gaussian variance bound]
	\label{lem:lower-gaussian-variance}
	Let $V$ be a standard normal random variable, and let
	$f\colon\R\to\R$ be bounded and Lipschitz. Then
	\[
		\bigl(\E[f'(V)]\bigr)^2
		\leq
		\Var\bigl(f(V)\bigr)\,,
	\]
	where $f'$ denotes the almost-everywhere derivative of $f$.
\end{lemma}

\section{Noise-trader loss and informed trader profit}
\label{sec:quantities}

We work in the normalised setting and restate the notation needed
below. Let $U$ and $V$ be independent standard normal random variables,
let $\phi\colon\R\to\R$ be Borel measurable, and set
\[
	Y=\phi(V)+U\,.
\]
Writing $\varrho$ for the standard normal density, define
$P\colon\R\to\R$ by
\[
	P(y)
	:=
	\frac{
		\displaystyle\int_\R
		v\varrho(v)\varrho\bigl(y-\phi(v)\bigr)\,\dd v
	}{
		\displaystyle\int_\R
		\varrho(v)\varrho\bigl(y-\phi(v)\bigr)\,\dd v
	},
	\qquad y\in\R\,.
\]
By Lemma~\ref{lem:P-properties}, $P$ is a continuous inverse
regression associated with $\phi$, and
\[
	P(Y)=\E[V\mid Y]
	\qquad\text{almost surely}\,.
\]
Define its Gaussian smoothing $F\colon\R\to\R$ by
\[
	F(x):=\E[P(x+U)],
	\qquad x\in\R\,.
\]
The properties of $F$ used below are in
Lemma~\ref{lem:F-properties}. All remaining notation is as in
Section~\ref{sec:preliminaries}.

\subsection{Noise-trader loss and covariance}

The \emph{noise traders' expected loss} is
\[
	\Gamma:=\E\!\left[U\bigl(P(Y)-V\bigr)\right]\,.
\]

\begin{lemma}
	\label{lem:ntloss}
	The loss $\Gamma$ is well defined and finite. For every Borel
	version $\widetilde P$ of the inverse regression
	$y\mapsto\E[V\mid Y=y]$,
	\[
		\Gamma
		=\E\!\left[U\bigl(\widetilde P(Y)-V\bigr)\right]
		=\E[U\widetilde P(Y)]
		=\Cov\bigl(U,\E[V\mid Y]\bigr)
		=\E[F'(\phi(V))]\,.
	\]
\end{lemma}

\begin{proof}
	By Lemma~\ref{lem:P-properties}, $P(Y)\in L^2$. Since
	$U,V\in L^2$, the Cauchy--Schwarz inequality gives
	$UP(Y),UV\in L^1$, so $\Gamma$ is well defined and finite. Since
	$\E[UV]=0$ and $P(Y)=\E[V\mid Y]$ almost surely,
	\[
		\Gamma
		=\E[UP(Y)]
		=\Cov\bigl(U,\E[V\mid Y]\bigr)\,.
	\]

	Conditioning on $V$, using the independence of $U$ and $V$, and
	applying \eqref{eq:F-derivative} give
	\[
		\E[UP(Y)\mid V]
		=
		\int_{\R}uP\bigl(\phi(V)+u\bigr)\varrho(u)\,\dd u
		=
		F'(\phi(V))
		\qquad\text{almost surely}\,.
	\]
	Since $UP(Y)$ is integrable, so is $F'(\phi(V))$. Taking
	expectations gives
	\[
		\Gamma=\E[F'(\phi(V))]\,.
	\]

	Finally, if $\widetilde P$ is any Borel version of the inverse
	regression, then
	\[
		\widetilde P(Y)=P(Y)\qquad\text{almost surely}\,.
	\]
	Hence
	\[
		U\bigl(\widetilde P(Y)-V\bigr)
		=
		U\bigl(P(Y)-V\bigr)
		\qquad\text{and}\qquad
		U\widetilde P(Y)=UP(Y)
		\qquad\text{almost surely}\,,
	\]
	which proves the remaining equalities.
\end{proof}

\subsection{Informed-trader profit}

The \emph{informed trader's conditional expected profit} is
\[
	\Pi(v)
	:=\phi(v)\bigl(v-F(\phi(v))\bigr),
	\qquad\text{for every }v\in\R\,.
\]

\begin{lemma}
	\label{lem:itprofit}
	The function $\Pi\colon\R\to\R$ is finite-valued and Borel
	measurable. For every Borel version $\widetilde P$ of the inverse
	regression and every $v\in\R$,
	\[
		\Pi(v)
		=
		\E\!\left[
			\bigl(v-\widetilde P(\phi(v)+U)\bigr)\phi(v)
			\right]\,.
	\]
\end{lemma}

\begin{proof}
	By Lemma~\ref{lem:F-properties}, $F$ is finite-valued and
	continuous. Since $\phi$ is Borel measurable, $\Pi$ is finite-valued
	and Borel measurable.

	For every Borel version $\widetilde P$ of the inverse regression,
	Lemma~\ref{lem:F-properties} gives, for every $v\in\R$,
	\[
		\widetilde P(\phi(v)+U)\in L^1,
		\qquad
		\E[\widetilde P(\phi(v)+U)]=F(\phi(v))\,.
	\]
	Therefore,
	\[
		\E\!\left[
			\bigl(v-\widetilde P(\phi(v)+U)\bigr)\phi(v)
			\right]
		=
		\phi(v)\bigl(v-F(\phi(v))\bigr)
		=
		\Pi(v)\,.
	\]
\end{proof}

The random variable $\Pi(V)$ need not be integrable, so the informed
trader's ex ante expected profit $\E[\Pi(V)]$ need not exist.
\section{Sharp universal bounds on noise-trader loss}
\label{sec:covariance-bound}

Retain the notation of Section~\ref{sec:quantities}. The following
proposition gives sharp universal bounds on noise-trader loss and
characterises the equality cases.

\begin{proposition}[Covariance bound and equality cases]
	\label{prop:covariance-maximizers}
	The quantity $\Gamma$ satisfies
	\begin{equation}
		-\frac12\leq\Gamma\leq\frac12\,.
		\label{eq:covariance-bound}
	\end{equation}
	Moreover, $\Gamma=\pm\tfrac12$ if and only if there exists
	$a\in\R$ such that
	\[
		\phi(v)=\pm v+a
		\qquad\text{for Lebesgue-almost every }v\in\R\,,
	\]
	with the same choice of sign.
\end{proposition}

\begin{proof}
	Set $M=\E[V\mid Y]$. Conditional Jensen's inequality and the tower
	property give
	\[
		\E[M^2]\leq\E\bigl[\E[V^2\mid Y]\bigr]=\E[V^2]=1\,.
	\]
	Thus, $U,M\in L^2$, and Cauchy--Schwarz gives $UM\in L^1$.
	By Lemma~\ref{lem:ntloss},
	\[
		\Gamma=\Cov(U,M)=\E[UM]\,,
	\]
	where the second equality uses $\E[U]=0$.

	Define
	\[
		X:=2M-V\,.
	\]
	Because $M,V\in L^2$, we have $X\in L^2$.
	For $s\in\{-1,1\}$,
	\[
		\begin{aligned}
			0 & \leq\E[(X-sU)^2]          \\
			  & =\E[X^2]+\E[U^2]-2s\E[XU] \\
			  & =1+\E[X^2]-2s\E[XU]\,.
		\end{aligned}
	\]
	By the preceding identity and the independence of $U$ and $V$,
	\[
		\E[XU]
		=2\E[MU]-\E[VU]
		=2\Gamma\,.
	\]
	Therefore,
	\[
		0\leq\E[(X-sU)^2]=1+\E[X^2]-4s\Gamma\,.
	\]
	We have
	\[
		\E[X^2]=4\E[M^2]-4\E[VM]+\E[V^2]\,.
	\]
	Because $V,M\in L^2$, Cauchy--Schwarz gives $VM\in L^1$.
	Moreover, because $M=\E[V\mid Y]$, the pull-out property gives
	\[
		\E[VM\mid Y]=M\E[V\mid Y]=M^2
		\qquad\text{almost surely}\,,
	\]
	and therefore
	\[
		\E[VM]=\E[M^2]\,.
	\]
	Consequently,
	\[
		\E[X^2]=\E[V^2]=1\,.
	\]
	Thus, for $s\in\{-1,1\}$,
	\[
		0\leq\E[(X-sU)^2]=2-4s\Gamma\,.
	\]
	Taking $s=1$ and $s=-1$ proves \eqref{eq:covariance-bound}.

	Fix $s\in\{-1,1\}$ and suppose that
	\[
		\Gamma=\frac{s}{2}\,.
	\]
	The preceding display gives
	\[
		\E[(X-sU)^2]=0\,,
	\]
	so $X=sU$ almost surely. Because $X=2M-V$ and $M=P(Y)$ almost
	surely by Lemma~\ref{lem:P-properties},
	\[
		2P(Y)=V+sU\qquad\text{almost surely}\,.
	\]
	Because $s^2=1$ and $Y=\phi(V)+U$, this implies
	\[
		Y-2sP(Y)=\phi(V)-sV\qquad\text{almost surely}\,.
	\]
	Because $(V,U)$ has an everywhere positive density, its law and
	Lebesgue measure on $\R^2$ have the same null sets. Consequently,
	\[
		\phi(v)+u-2sP\bigl(\phi(v)+u\bigr)=\phi(v)-sv
	\]
	for Lebesgue-almost every $(v,u)\in\R^2$.

	Let $\mathcal L^1$ and $\mathcal L^2$ denote Lebesgue measure on
	$\R$ and $\R^2$, respectively. Because $\phi$ and $P$ are Borel
	measurable, the exceptional set
	\[
		N:=\left\{(v,u)\in\R^2:
		\phi(v)+u-2sP\bigl(\phi(v)+u\bigr)\ne\phi(v)-sv\right\}
	\]
	is Borel measurable. The preceding almost-everywhere identity gives
	\[
		\mathcal L^2(N)=0\,.
	\]

	For $v\in\R$, let $N_v:=\{u\in\R:(v,u)\in N\}$ be the vertical
	section of $N$ at $v$. Tonelli's theorem gives
	\[
		0=\mathcal L^2(N)=\int_\R\mathcal L^1(N_v)\,\dd v\,.
	\]
	Because the integrand is nonnegative, there exists a set
	$A\subseteq\R$ of full Lebesgue measure such that
	\[
		\mathcal L^1(N_v)=0\qquad\text{for every }v\in A\,.
	\]

	Fix $v\in A$ and define
	\[
		C_v:=\{y\in\R:y-2sP(y)\ne\phi(v)-sv\}\,.
	\]
	We have
	\[
		C_v=\phi(v)+N_v:=\{\phi(v)+u:u\in N_v\}\,.
	\]
	Indeed, writing $y=\phi(v)+u$, the inequality defining $C_v$ is
	equivalent to $(v,u)\in N$. Because Lebesgue measure is invariant
	under translations,
	\[
		\mathcal L^1(C_v)=\mathcal L^1(N_v)=0\,.
	\]
	Consequently, for every $v\in A$,
	\[
		y-2sP(y)=\phi(v)-sv
		\qquad\text{for Lebesgue-almost every }y\in\R\,.
	\]

	Let $v_1,v_2\in A$. Both $C_{v_1}$ and $C_{v_2}$ are null, so
	their union is null. Choose
	\[
		y\in\R\setminus\bigl(C_{v_1}\cup C_{v_2}\bigr)\,.
	\]
	We have
	\[
		\phi(v_1)-sv_1=y-2sP(y)=\phi(v_2)-sv_2\,.
	\]
	Thus, $\phi(v)-sv$ is constant on $A$. Because $A$ has full
	Lebesgue measure, there exists $a\in\R$ such that
	\[
		\phi(v)=sv+a
		\qquad\text{for Lebesgue-almost every }v\in\R\,.
	\]

	Conversely, suppose that, for some $s\in\{-1,1\}$ and $a\in\R$,
	\[
		\phi(v)=sv+a
		\qquad\text{for Lebesgue-almost every }v\in\R\,.
	\]
	Because the law of $V$ is absolutely continuous,
	\[
		Y=sV+U+a\qquad\text{almost surely}\,.
	\]
	The pair $(V,Y)$ is jointly Gaussian, with
	\[
		\E[Y]=a,\qquad \Cov(V,Y)=s,\qquad \Var(Y)=2\,.
	\]
	The Gaussian regression formula gives
	\[
		M=\E[V\mid Y]
		=\frac{s}{2}(Y-a)
		=\frac{V+sU}{2}
		\qquad\text{almost surely}\,.
	\]
	It follows that
	\[
		\Gamma
		=\E[UM]
		=\frac12\bigl(\E[UV]+s\E[U^2]\bigr)
		=\frac{s}{2}\,.
	\]
	This proves the equality characterisation for both choices of sign.
\end{proof}

\section{Accounting identity for monotone strategies}
\label{sec:balance}

Retain the notation of Section~\ref{sec:quantities}. This section
establishes an accounting identity between the noise traders' expected
loss and the informed trader's ex ante expected profit.

The proof uses a conditionally independent replica of $V$ to express
$2\Gamma$ as the expectation of a symmetric supermodularity increment
of the bilinear map $(x,v)\mapsto xv$. This increment is nonnegative
when $\phi$ is nondecreasing. Lattice truncation and two applications
of monotone convergence avoid any a priori integrability assumption
on either $\phi(V)$ or $\Pi(V)$.

\begin{proposition}[Accounting identity]
	\label{prop:balance}
	If $\phi$ is nondecreasing and
	\[
		\Pi(V)\geq 0
		\qquad\text{almost surely}\,,
	\]
	then $\Pi(V)\in L^1$ and
	\[
		\Gamma=\E[\Pi(V)]\,.
	\]
\end{proposition}

\begin{proof}
	Let $V^{(0)}:=V$, and let $V^{(1)}$ be the conditionally
	independent replica supplied by Lemma~\ref{lem:resampling}. For
	$i\in\{0,1\}$, define
	\[
		U^{(i)}:=Y-\phi(V^{(i)})\,.
	\]
	By definition, $U^{(0)}=U$. By Lemma~\ref{lem:resampling}
	\ref{item:resampling-marginal},
	\[
		(V^{(1)},Y)\overset{d}{=}(V^{(0)},Y)\,.
	\]
	Therefore,
	\[
		(V^{(1)},U^{(1)})
		\overset{d}{=}
		(V^{(0)},U^{(0)})\,.
	\]
	In particular, $V^{(1)},U^{(1)}\in L^2$. Hence, by
	Cauchy--Schwarz, all four products
	\[
		U^{(0)}V^{(1)},\qquad
		U^{(1)}V^{(0)},\qquad
		U^{(1)}V^{(1)},\qquad
		U^{(0)}V^{(0)}
	\]
	belong to $L^1$.

	\medskip

	\emph{Part I. Symmetric representation of $\Gamma$.}
	Because $U^{(0)}=U$ and $V^{(0)}=V$ are independent and centred,
	\[
		\E[U^{(0)}V^{(0)}]=0\,.
	\]
	By Lemma~\ref{lem:ntloss},
	\[
		\Gamma=\E[U^{(0)}P(Y)]\,.
	\]

	By Lemma~\ref{lem:resampling}
	\ref{item:resampling-exchangeable},
	\[
		\E[U^{(1)}V^{(0)}]
		=\E[U^{(0)}V^{(1)}]\,.
	\]
	Moreover,
	\[
		U^{(0)}=Y-\phi(V^{(0)})
	\]
	is measurable with respect to the sigma-algebra generated by
	$(Y,V^{(0)})$. Because $U^{(0)}V^{(1)}\in L^1$, the tower and
	pull-out properties give
	\[
		\begin{aligned}
			\E[U^{(1)}V^{(0)}]
			 & =
			\E[U^{(0)}V^{(1)}] \\
			 & =
			\E\!\left[
			\E[U^{(0)}V^{(1)}\mid Y,V^{(0)}]
			\right]            \\
			 & =
			\E\!\left[
			U^{(0)}
			\E[V^{(1)}\mid Y,V^{(0)}]
			\right]            \\
			 & =
			\E[U^{(0)}P(Y)]    \\
			 & =
			\Gamma\,,
		\end{aligned}
	\]
	where the penultimate equality follows from
	Lemma~\ref{lem:resampling}
	\ref{item:resampling-posterior-mean}.

	The equality in law of
	$(V^{(1)},U^{(1)})$ and $(V^{(0)},U^{(0)})$ gives
	\[
		\E[U^{(1)}V^{(1)}]
		=\E[U^{(0)}V^{(0)}]
		=0\,.
	\]
	Because
	\[
		U^{(1)}-U^{(0)}
		=\phi(V^{(0)})-\phi(V^{(1)})\,,
	\]
	we obtain
	\begin{align}
		2\Gamma
		 & =
		\E[U^{(1)}V^{(0)}]
		+\E[U^{(0)}V^{(1)}]
		-\E[U^{(1)}V^{(1)}]
		-\E[U^{(0)}V^{(0)}]
		\nonumber \\
		 & =
		\E\!\left[
		\bigl(U^{(1)}-U^{(0)}\bigr)
		\bigl(V^{(0)}-V^{(1)}\bigr)
		\right]
		\nonumber \\
		 & =
		\E\!\left[
		\bigl(\phi(V^{(0)})-\phi(V^{(1)})\bigr)
		\bigl(V^{(0)}-V^{(1)}\bigr)
		\right]\,.
		\label{eq:exchangeable-balance}
	\end{align}

	\medskip

	\emph{Part II. Truncation.}
	For every integer $n\geq1$, define
	\[
		\phi_n(v):=(-n)\vee\bigl(\phi(v)\wedge n\bigr)\,.
	\]
	By Lemma~\ref{lem:resampling}
	\ref{item:resampling-smoothed-mean},
	\[
		\E[V^{(1)}\mid V^{(0)}]
		=F\bigl(\phi(V^{(0)})\bigr)
		\qquad\text{almost surely}\,.
	\]
	Because $V^{(1)}\in L^1$, it follows that
	\[
		F\bigl(\phi(V^{(0)})\bigr)\in L^1\,.
	\]
	Because $\phi_n$ is bounded, all terms in the following calculation
	are integrable. Expanding the product and using
	Lemma~\ref{lem:resampling}
	\ref{item:resampling-exchangeable} give
	\begin{align}
		 & \frac12\E\!\left[
		\bigl(\phi_n(V^{(0)})-\phi_n(V^{(1)})\bigr)
		\bigl(V^{(0)}-V^{(1)}\bigr)
		\right]
		\nonumber            \\
		 & \qquad
		=\E\!\left[
		\phi_n(V^{(0)})
		\bigl(V^{(0)}-V^{(1)}\bigr)
		\right]
		\nonumber            \\
		 & \qquad
		=\E\!\left[
			\phi_n(V^{(0)})
			\left(
			V^{(0)}
			-\E[V^{(1)}\mid V^{(0)}]
			\right)
			\right]
		\nonumber            \\
		 & \qquad
		=\E\!\left[
			\phi_n(V^{(0)})
			\bigl(
			V^{(0)}-F(\phi(V^{(0)}))
			\bigr)
			\right]\,.
		\label{eq:truncated-balance}
	\end{align}

	\medskip

	\emph{Part III. First monotone convergence.}
	By hypothesis, $\phi$ is nondecreasing. Because the truncation map is
	nondecreasing, each $\phi_n$ is also nondecreasing. Hence, for every
	$n\geq1$,
	\[
		\begin{aligned}
			 & \bigl(
			\phi_n(V^{(0)})-\phi_n(V^{(1)})
			\bigr)
			\bigl(V^{(0)}-V^{(1)}\bigr) \\
			 & \qquad
			=
			\bigl|
			\phi_n(V^{(0)})-\phi_n(V^{(1)})
			\bigr|
			\,\bigl|V^{(0)}-V^{(1)}\bigr|
			\geq0\,.
		\end{aligned}
	\]

	For each fixed outcome,
	\[
		\bigl|
		\phi_n(V^{(0)})-\phi_n(V^{(1)})
		\bigr|
	\]
	is the Lebesgue length of the intersection of $[-n,n]$ with the
	interval having endpoints $\phi(V^{(0)})$ and $\phi(V^{(1)})$.
	As $n$ increases, the intervals $[-n,n]$ expand to $\R$.
	Therefore,
	\[
		\bigl|
		\phi_n(V^{(0)})-\phi_n(V^{(1)})
		\bigr|
		\uparrow
		\bigl|
		\phi(V^{(0)})-\phi(V^{(1)})
		\bigr|
	\]
	almost surely. Because $\lvert V^{(0)}-V^{(1)}\rvert$ is
	nonnegative and does not depend on $n$, it follows that
	\[
		\begin{aligned}
			0
			 & \leq
			\bigl(\phi_n(V^{(0)})-\phi_n(V^{(1)})\bigr)
			\bigl(V^{(0)}-V^{(1)}\bigr) \\
			 & \uparrow
			\bigl(\phi(V^{(0)})-\phi(V^{(1)})\bigr)
			\bigl(V^{(0)}-V^{(1)}\bigr)
			\qquad\text{almost surely}\,.
		\end{aligned}
	\]

	Hence the monotone convergence theorem and
	\eqref{eq:exchangeable-balance} give
	\begin{equation}
		\begin{aligned}
			 & \lim_{n\to\infty}\frac12
			\E\!\left[
			\bigl(
			\phi_n(V^{(0)})-\phi_n(V^{(1)})
			\bigr)
			\bigl(V^{(0)}-V^{(1)}\bigr)
			\right]                     \\
			 & \qquad
			=
			\frac12
			\E\!\left[
			\bigl(
			\phi(V^{(0)})-\phi(V^{(1)})
			\bigr)
			\bigl(V^{(0)}-V^{(1)}\bigr)
			\right]
			=\Gamma\,.
		\end{aligned}
		\label{eq:first-monotone-convergence}
	\end{equation}
	In particular, this limit is finite.

	\medskip

	\emph{Part IV. Matching monotone convergence.}
	Because $V^{(0)}=V$, the hypothesis gives
	\[
		\Pi(V^{(0)})\geq0
		\qquad\text{almost surely}\,.
	\]
	By the definition of $\Pi$,
	\[
		\Pi(V^{(0)})
		=
		\phi(V^{(0)})
		\bigl(
		V^{(0)}-F(\phi(V^{(0)}))
		\bigr)\,.
	\]

	On a full-probability event on which
	$\Pi(V^{(0)})\geq0$, consider the following three cases. If
	$\phi(V^{(0)})>0$, then
	\[
		V^{(0)}-F(\phi(V^{(0)}))\geq0
	\]
	and
	\[
		\phi_n(V^{(0)})\uparrow\phi(V^{(0)})\,.
	\]
	If $\phi(V^{(0)})<0$, then
	\[
		V^{(0)}-F(\phi(V^{(0)}))\leq0
	\]
	and
	\[
		\phi_n(V^{(0)})\downarrow\phi(V^{(0)})\,.
	\]
	Multiplication by the nonpositive quantity
	$V^{(0)}-F(\phi(V^{(0)}))$ reverses the latter order. If
	$\phi(V^{(0)})=0$, then both $\phi_n(V^{(0)})$ and
	$\Pi(V^{(0)})$ are zero. Consequently,
	\[
		0
		\leq
		\phi_n(V^{(0)})
		\bigl(
		V^{(0)}-F(\phi(V^{(0)}))
		\bigr)
		\uparrow
		\Pi(V^{(0)})
		\qquad\text{almost surely}\,.
	\]

	The monotone convergence theorem, with
	$\E[\Pi(V^{(0)})]$ initially understood as an extended expectation in
	$[0,\infty]$, therefore gives
	\[
		\begin{aligned}
			\E[\Pi(V)]
			 & =\E[\Pi(V^{(0)})] \\
			 & =
			\lim_{n\to\infty}
			\E\!\left[
				\phi_n(V^{(0)})
				\bigl(
				V^{(0)}-F(\phi(V^{(0)}))
				\bigr)
			\right]              \\
			 & =
			\frac12\lim_{n\to\infty}
			\E\!\left[
			\bigl(
			\phi_n(V^{(0)})-\phi_n(V^{(1)})
			\bigr)
			\bigl(V^{(0)}-V^{(1)}\bigr)
			\right]              \\
			 & =\Gamma\,.
		\end{aligned}
	\]
	Here the first equality uses $V^{(0)}=V$, the second follows from the
	monotone convergence theorem, the third follows from
	\eqref{eq:truncated-balance}, and the last follows from
	\eqref{eq:first-monotone-convergence}. Because the limit in
	Part~III is finite and $\Pi(V)\geq0$ almost surely,
	\[
		\E[\lvert\Pi(V)\rvert]
		=\E[\Pi(V)]
		=\Gamma
		<\infty\,.
	\]
	Thus, $\Pi(V)\in L^1$.
\end{proof}

\section{Noise-trader loss saturation}
\label{sec:covariance-saturation}

Retain the notation of Section~\ref{sec:quantities}. The following
proposition shows that, at equilibrium, the noise traders' expected
loss attains its sharp upper bound. Throughout this section, $P$
denotes the continuous inverse regression associated with $\phi$.
By the version clause of Lemma~\ref{lem:F-properties},
$(\phi,P)$ is an equilibrium if and only if
$(\phi,\widetilde P)$ is an equilibrium for any other Borel version
$\widetilde P$ of the inverse regression. By
Lemma~\ref{lem:ntloss}, the value of $\Gamma$ is also independent of
the chosen version.

\begin{proposition}[Covariance saturation]
	\label{prop:covariance-saturation}
	If $(\phi,P)$ is an equilibrium in the normalised model, then
	\[
		\Gamma=\frac12\,.
	\]
\end{proposition}

\begin{proof}
	By Lemma~\ref{lem:F-properties}, for every $x\in\R$,
	\[
		P(x+U)\in L^1\,,
		\qquad
		\E[P(x+U)]=F(x)\,.
	\]
	Hence, for every $v,x\in\R$,
	\[
		\E\bigl[(v-P(x+U))x\bigr]
		=x\bigl(v-F(x)\bigr)\,.
	\]
	Profit maximisation therefore gives
	\[
		\phi(v)\in\argmax_{x\in\R}x\bigl(v-F(x)\bigr)
		\qquad\text{for every }v\in\R\,.
	\]

	\medskip

	Lemma~\ref{lem:foc} shows that $\phi$ is strictly increasing. It
	also shows that $W$ is real-valued and locally absolutely
	continuous, with
	\[
		W'(v)=\phi(v)
		\qquad\text{for Lebesgue-almost every }v\in\R\,,
	\]
	and
	\[
		W(v)=\phi(v)^2F'(\phi(v))
		\qquad\text{for every }v\in\R\,.
	\]

	By the definition of $\Pi$ and pointwise maximisation,
	\[
		\Pi(v)
		=\phi(v)\bigl(v-F(\phi(v))\bigr)
		=W(v)
		\qquad\text{for every }v\in\R\,.
	\]
	Because $\phi$ is nondecreasing and
	\[
		\Pi(v)=W(v)\geq0
		\qquad\text{for every }v\in\R\,,
	\]
	Proposition~\ref{prop:balance} applies. It gives
	$\Pi(V)\in L^1$ and
	\[
		\Gamma
		=\E[\Pi(V)]
		=\E[W(V)]\,.
	\]
	In particular,
	\[
		W(V)\in L^1
		\qquad\text{and}\qquad
		\Gamma<\infty\,.
	\]

	\medskip

	Lemma~\ref{lem:posterior} gives
	\[
		F'(x)>0
		\qquad\text{for every }x\in\R\,.
	\]
	Define
	\[
		q(v):=F'(\phi(v))
		\qquad\text{for every }v\in\R\,.
	\]
	Then $q\colon\R\to(0,\infty)$ is Borel measurable, and
	\[
		W(v)=\phi(v)^2q(v)
		\qquad\text{for every }v\in\R\,.
	\]

	By Lemma~\ref{lem:ntloss},
	\[
		q(V)\in L^1
		\qquad\text{and}\qquad
		\Gamma=\E[q(V)]\,.
	\]
	Because $q(V)>0$ almost surely, $\Gamma>0$.
	\medskip

	Set
	\[
		c:=F(0)\,.
	\]
	Because $F$ is strictly increasing,
	\[
		x\bigl(c-F(x)\bigr)<0
		\qquad\text{for every }x\ne 0\,,
	\]
	whereas the expression vanishes at $x=0$. Thus, $0$ is the unique
	maximiser of
	\[
		x\longmapsto x\bigl(c-F(x)\bigr)\,.
	\]
	The pointwise maximisation condition at $v=c$ therefore gives
	\[
		\phi(c)=0\,.
	\]
	Because $\phi$ is strictly increasing,
	\[
		\phi(v)<0
		\qquad\text{for every }v<c\,,
	\]
	and
	\[
		\phi(v)>0
		\qquad\text{for every }v>c\,.
	\]

	The identity $W=\phi^2q$ gives
	\[
		W(c)=0
	\]
	and
	\[
		W(v)>0
		\qquad\text{for every }v\ne c\,.
	\]
	Define $h\colon\R\to\R$ by
	\[
		h(v):=
		\begin{cases}
			-\sqrt{2W(v)}\,, & v<c\,, \\
			0\,,             & v=c\,, \\
			\sqrt{2W(v)}\,,  & v>c\,.
		\end{cases}
	\]

	Let $J$ be a nonempty compact interval contained in either
	$(-\infty,c)$ or $(c,\infty)$, and let $s\in\{-1,1\}$ be the
	constant sign of $\phi$ on $J$. We have
	\[
		W(v)=\phi(v)^2q(v)>0
		\qquad\text{for every }v\in J\,.
	\]
	Because $W$ is absolutely continuous on $J$ and $2W(J)$ is a
	compact subset of $(0,\infty)$, on which the square root is
	Lipschitz, the function $h=s\sqrt{2W}$ is absolutely continuous
	on $J$. Using $W'=\phi$ almost everywhere, the chain rule gives
	\[
		\begin{aligned}
			h'(v)
			 & =
			\frac{s\phi(v)}{\sqrt{2W(v)}}  \\
			 & =
			\frac{|\phi(v)|}{\sqrt{2W(v)}} \\
			 & =
			\frac{1}{\sqrt{2q(v)}}
		\end{aligned}
	\]
	for Lebesgue-almost every $v\in J$.

	Define
	\[
		r(v):=\frac{1}{\sqrt{2q(v)}}
		\qquad\text{for every }v\in\R\,.
	\]
	Then $r$ is Borel measurable and everywhere positive. We next show
	that $h$ is locally absolutely continuous, including at $c$, and
	that $r$ is a representative of its almost-everywhere derivative.
	Let $K=[a,b]$ be a compact interval containing $c$. Because $\phi$ is
	increasing and $F'$ is continuous and everywhere positive,
	\[
		m_K:=\min_{x\in[\phi(a),\phi(b)]}F'(x)>0\,.
	\]
	Thus,
	\[
		q(v)\geq m_K
		\qquad\text{and}\qquad
		r(v)\leq\frac{1}{\sqrt{2m_K}}
		\qquad\text{for every }v\in K\,.
	\]
	The preceding absolute-continuity and derivative calculation show
	that $h$ is $(2m_K)^{-1/2}$-Lipschitz on each compact subinterval
	of $K\setminus\{c\}$. Moreover, $W$ is continuous and $W(c)=0$, so
	$h$ is continuous at $c$. Letting an endpoint tend to $c$, and then
	using the triangle inequality for points on opposite sides of $c$,
	shows that $h$ is $(2m_K)^{-1/2}$-Lipschitz on $K$. Hence $h$ is
	locally absolutely continuous on $\R$, and
	\[
		h'(v)=r(v)
	\]
	for Lebesgue-almost every $v\in\R$. Because $r>0$, the function $h$
	is strictly increasing.

	Moreover,
	\[
		h(v)^2=2W(v)
		\qquad\text{for every }v\in\R\,,
	\]
	and
	\[
		\frac{1}{r(v)^2}=2q(v)
		\qquad\text{for every }v\in\R\,.
	\]
	Consequently,
	\[
		\E[h(V)^2]
		=2\E[W(V)]
		=2\Gamma<\infty
	\]
	and
	\[
		\E\!\left[\frac{1}{r(V)^2}\right]
		=2\E[q(V)]
		=2\Gamma<\infty\,.
	\]

	For every integer $N>|c|$, define
	\[
		T_N(v):=(-N)\vee(v\wedge N),
		\qquad
		h_N:=h\circ T_N\,.
	\]
	The function $h_N$ is bounded and Lipschitz, and
	\[
		h_N'(v)=r(v)\mathbf 1_{(-N,N)}(v)
	\]
	for Lebesgue-almost every $v\in\R$.
	Lemma~\ref{lem:lower-gaussian-variance} therefore gives
	\[
		\left(
		\E\!\left[r(V)\mathbf 1_{\{|V|<N\}}\right]
		\right)^2
		\leq
		\Var\bigl(h_N(V)\bigr)\,.
	\]
	Because $h$ is increasing and $h(c)=0$, we have
	$|h_N(v)|\leq|h(v)|$ for every $v\in\R$ and every $N>|c|$.
	Thus, $h_N(V)\to h(V)$ in $L^2$, and hence
	$\Var(h_N(V))\to\Var(h(V))$. Letting $N\to\infty$ and using
	monotone convergence on the left gives
	\[
		\bigl(\E[r(V)]\bigr)^2
		\leq
		\Var\bigl(h(V)\bigr)\,.
	\]
	In particular, $r(V)\in L^1$.

	Moreover, Cauchy--Schwarz and the preceding moment bound give
	\[
		\E\!\left[\frac{1}{r(V)}\right]
		\leq
		\sqrt{\E\!\left[\frac{1}{r(V)^2}\right]}
		<\infty\,.
	\]
	Combining the preceding variance bound with Cauchy--Schwarz now gives
	\[
		\begin{aligned}
			1
			 & =
			\left(
			\E\!\left[
				\sqrt{r(V)}\frac{1}{\sqrt{r(V)}}
				\right]
			\right)^2                                \\
			 & \leq
			\E[r(V)]\,
			\E\!\left[\frac{1}{r(V)}\right]          \\
			 & \leq
			\sqrt{\Var\bigl(h(V)\bigr)}
			\sqrt{\E\!\left[\frac{1}{r(V)^2}\right]} \\
			 & \leq
			\sqrt{
				\E[h(V)^2]
				\E\!\left[\frac{1}{r(V)^2}\right]
			}                                        \\
			 & =
			2\sqrt{\E[W(V)]\,\E[q(V)]}
			=2\Gamma\,.
		\end{aligned}
	\]
	Thus, $\Gamma\geq1/2$. The reverse inequality follows from
	\eqref{eq:covariance-bound}. Therefore,
	\[
		\Gamma=\frac12\,.
	\]
\end{proof}

\section{Rigidity and equilibrium uniqueness}
\label{sec:rigidity}

The well-known Gaussian regression formula shows that the identity map satisfies
the pointwise maximisation condition below. The content of
the next theorem is the converse: among all Borel-measurable maps, this
condition forces the identity.

\begin{theorem}[Rigidity]
	\label{thm:rigidity}
	Let $V$ and $U$ be independent standard normal random variables.
	For every Borel-measurable map $\phi\colon\R\to\R$, let $P_\phi\colon\R\to\R$
	be a Borel version of the inverse regression
	\[
		P_\phi(y)=\E[V\mid\phi(V)+U=y]\,,
	\]
	and let
	\[
		F_\phi(x):=\E[P_\phi(x+U)]\,.
	\]
	The pointwise maximisation condition
	\[
		\phi(v)\in\argmax_{x\in\R}
		\bigl(xv-xF_\phi(x)\bigr)
		\qquad\text{for every }v\in\R
	\]
	holds if and only if $\phi=\id$.
\end{theorem}

\begin{proof}
	Suppose first that the pointwise maximisation condition holds. Let
	$P$ be the continuous inverse regression associated with $\phi$.
	By the version clause of Lemma~\ref{lem:F-properties},
	\[
		\E[P(x+U)]=F_\phi(x)
		\qquad\text{for every }x\in\R\,.
	\]
	Thus, the pointwise maximisation condition and market efficiency
	show that $(\phi,P)$ is an equilibrium of the standardised model.
	By Lemma~\ref{lem:foc}, the map $\phi$ is strictly increasing, while
	Propositions~\ref{prop:covariance-saturation}
	and~\ref{prop:covariance-maximizers} give
	\[
		\phi(v)=v+a
		\qquad\text{for Lebesgue-almost every }v\in\R
	\]
	for some $a\in\R$. Because $\phi$ is increasing and agrees almost
	everywhere with the continuous function $v\mapsto v+a$, this
	equality holds for every $v\in\R$.

	Gaussian regression now gives
	\[
		F_\phi(x)=\frac{x-a}{2}
		\qquad\text{for every }x\in\R\,.
	\]
	Hence the unique maximiser of
	\[
		x\bigl(v-F_\phi(x)\bigr)
		=
		xv-\frac{x^2}{2}+\frac{a}{2}x
	\]
	is $v+a/2$. Pointwise maximisation therefore gives
	$v+a=v+a/2$ for every $v$, so $a=0$ and $\phi=\id$.

	Conversely, if $\phi=\id$, then Gaussian regression gives
	$F_\phi(x)=x/2$, and $x(v-x/2)$ is uniquely maximised at $x=v$.
\end{proof}
We now prove the main theorem. It is a translation and rescaling of the
rigidity theorem.

\begin{proof}[Proof of Theorem~\ref{thm:kyle}]
	Define
	\[
		\widehat V:=\frac{V-\mu}{\tau},
		\qquad
		\widehat U:=\frac{U}{\sigma},
		\qquad
		\widehat\phi(t):=\frac{\phi(\mu+\tau t)}{\sigma},
		\qquad
		\widehat P(y):=\frac{P(\sigma y)-\mu}{\tau}\,.
	\]
	Then $\widehat V$ and $\widehat U$ are independent standard normal
	random variables, and
	\[
		\widehat\phi(\widehat V)+\widehat U
		=\frac{Y_\phi}{\sigma}\,.
	\]
	A direct change of variables shows that $(\phi,P)$ is an equilibrium
	if and only if $(\widehat\phi,\widehat P)$ is an equilibrium of the
	standardised model: market efficiency is preserved, and expected
	profits are multiplied by the positive constant $\sigma\tau$.

	By Theorem~\ref{thm:rigidity} and Gaussian regression, this holds if
	and only if
	\[
		\widehat\phi=\id
		\qquad\text{and}\qquad
		\widehat P\bigl(\widehat\phi(\widehat V)+\widehat U\bigr)
		=
		\frac{\widehat\phi(\widehat V)+\widehat U}{2}
		\quad\text{almost surely}\,.
	\]
	Undoing the normalisation gives
	\[
		\phi(v)=\frac{\sigma}{\tau}(v-\mu)
		\qquad\text{for every }v\in\R
	\]
	and
	\[
		P(Y_\phi)
		=
		\mu+\frac{\tau}{2\sigma}Y_\phi
		\qquad\text{almost surely}\,.
	\]
\end{proof}

\begingroup
\renewcommand{\addcontentsline}[3]{}
\begin{acks}[Acknowledgments]
	We gratefully acknowledge Thomas Anton, who worked as a research
	assistant at the Australian National University in 2022-2023, for
	assistance with two preliminary aspects of the paper: the regularity
	of $P$ and $F$, and the conditional-replica construction based on
	Kallenberg's transfer theorem.
\end{acks}
\endgroup

\appendix

\section{Proofs of the preliminary lemmas}
\label{app:proofs-preliminaries}

\begin{proof}[Proof of Lemma~\ref{lem:densities}]
	Because $U$ and $V$ are independent, the pair $(V,U)$ has density
	$\varrho(v)\varrho(u)$. For every bounded Borel-measurable function
	$g\colon\R^2\to\R$, substitution of $y=\phi(v)+u$ at fixed $v$
	gives
	\[
		\begin{aligned}
			\E[g(V,Y)]
			 & =\int_\R\int_\R
			g\bigl(v,\phi(v)+u\bigr)\varrho(v)\varrho(u)
			\,\dd u\,\dd v     \\
			 & =\int_\R\int_\R
			g(v,y)\varrho(v)\varrho\bigl(y-\phi(v)\bigr)
			\,\dd y\,\dd v\,.
		\end{aligned}
	\]
	This identifies the joint density. Integrating out $v$ and applying
	Tonelli's theorem gives the density $p_Y$ of $Y$. Moreover, the
	integrand defining $p_Y$ is everywhere positive and bounded by
	$(2\pi)^{-1/2}\varrho(v)$. Thus, for every $y\in\R$,
	\[
		0<p_Y(y)
		\leq\frac{1}{\sqrt{2\pi}}\int_\R\varrho(v)\,\dd v
		=\frac{1}{\sqrt{2\pi}}\,;
	\]
	in particular, $p_Y$ is finite and everywhere positive.
\end{proof}

\begin{proof}[Proof of Lemma~\ref{lem:P-properties}]
	Define, locally in this proof,
	\[
		n(y):=\int_\R v\varrho(v)\varrho\bigl(y-\phi(v)\bigr)\,\dd v,
		\qquad
		q(y):=\int_\R|v|\varrho(v)
		\varrho\bigl(y-\phi(v)\bigr)\,\dd v\,.
	\]
	The integrands of $n$ and $q$ are dominated by
	$(2\pi)^{-1/2}|v|\varrho(v)$, and that of $p_Y$ by
	$(2\pi)^{-1/2}\varrho(v)$, uniformly in $y$. Because $\varrho$ is
	continuous, dominated convergence shows that $n$, $q$, and $p_Y$
	are continuous. Because $p_Y$ is everywhere positive, $P=n/p_Y$
	is continuous.

	Let $g\colon\R\to\R$ be bounded and Borel measurable. Because the
	joint law of $(V,Y)$ has the density in
	Lemma~\ref{lem:densities} and
	\[
		\int_\R\int_\R|vg(y)|\varrho(v)
		\varrho\bigl(y-\phi(v)\bigr)\,\dd y\,\dd v
		\leq\sup|g|\,\E[|V|]<\infty\,,
	\]
	Fubini's theorem gives
	\[
		\E[Vg(Y)]
		=\int_\R g(y)n(y)\,\dd y
		=\int_\R g(y)P(y)p_Y(y)\,\dd y
		=\E\bigl[g(Y)P(Y)\bigr]\,.
	\]
	Moreover, Tonelli's theorem gives
	$\int_\R q(y)\,\dd y=\E[|V|]$, so
	\[
		\E\bigl[|P(Y)|\bigr]
		=\int_\R|n(y)|\,\dd y
		\leq\int_\R q(y)\,\dd y
		<\infty\,.
	\]
	Thus, $P(Y)=\E[V\mid Y]$ almost surely. Conditional Jensen's inequality
	gives
	\[
		\E[P(Y)^2]\leq\E[V^2]=1\,.
	\]
\end{proof}

\begin{proof}[Proof of Lemma~\ref{lem:F-properties}]
	We first prove that $P(x+U)\in L^1$ for every $x\in\R$. Define,
	locally in this proof,
	\[
		q(y):=\int_\R|v|\varrho(v)
		\varrho\bigl(y-\phi(v)\bigr)\,\dd v
		\qquad\text{for every }y\in\R\,.
	\]
	Let $\nu$ be the law of $\phi(V)$, so that
	\[
		p_Y(y)=\int_\R\varrho(y-a)\,\nu(\dd a)\,.
	\]
	For fixed $x\in\R$, define
	\[
		K_x(a):=\int_\R
		\frac{\varrho(y-x)\varrho(y-a)}{p_Y(y)}\,\dd y
		\qquad\text{for every }a\in\R\,.
	\]
	Because $|P|\leq q/p_Y$, Tonelli's theorem gives
	\[
		\begin{aligned}
			\E\bigl[|P(x+U)|\bigr]
			 & =\int_\R|P(y)|\varrho(y-x)\,\dd y \\
			 & \leq\int_\R\frac{q(y)}{p_Y(y)}
			\varrho(y-x)\,\dd y                  \\
			 & =\E\bigl[|V|K_x(\phi(V))\bigr]\,.
		\end{aligned}
	\]
	We claim that $K_x$ is bounded outside a $\nu$-null set. Suppose
	first that $\nu((x,\infty))>0$. Choose $x<\alpha<\beta$ such that
	$r=\nu([\alpha,\beta])>0$, and put $m=(\alpha+\beta)/2$. For
	$y\geq m$ and $a\in[\alpha,\beta]$, we have
	$|y-a|\leq|y-\alpha|$, so
	$p_Y(y)\geq r\varrho(y-\alpha)$. Similarly,
	$p_Y(y)\geq r\varrho(y-\beta)$ for $y<m$. Thus, for every
	$a\geq x$,
	\[
		K_x(a)\leq\frac1r\sum_{d\in\{\alpha,\beta\}}
		\int_\R\frac{\varrho(y-x)\varrho(y-a)}{\varrho(y-d)}
		\,\dd y\,.
	\]
	Completing the square gives
	\[
		\int_\R\frac{\varrho(y-x)\varrho(y-a)}{\varrho(y-d)}
		\,\dd y=e^{(x-d)(a-d)}\,.
	\]
	Because $d>x$ and $a\geq x$, the exponent is at most $(x-d)^2$.
	Thus, $K_x$ is bounded on $[x,\infty)$. The symmetric argument gives boundedness on $(-\infty,x]$ whenever
	$\nu((-\infty,x))>0$. If neither open half-line has positive
	$\nu$-measure, then $\nu=\delta_x$ and $K_x(x)=1$. This proves the
	claim. Consequently, $K_x(\phi(V))$ is almost surely bounded, and
	\[
		\E\bigl[|P(x+U)|\bigr]\leq\E\bigl[|V|K_x(\phi(V))\bigr]<\infty\,.
	\]
	Because $x+U$ has density $y\mapsto\varrho(y-x)$, the definition
	\eqref{eq:F-definition} reads
	\[
		F(x)=\int_\R P(y)\varrho(y-x)\,\dd y\,.
	\]
	This proves \eqref{eq:F-convolution}.

	If $\widetilde P$ is as in the statement, then
	$\widetilde P(Y)=P(Y)$ almost surely, so $\widetilde P=P$ almost
	everywhere under the law of $Y$. Because $p_Y>0$ everywhere by
	Lemma~\ref{lem:densities}, the equality holds Lebesgue-almost
	everywhere. The law of $x+U$ is absolutely continuous, so
	$\widetilde P(x+U)=P(x+U)$ almost surely. In particular,
	$\widetilde P(x+U)\in L^1$ and $\E[\widetilde P(x+U)]=F(x)$.

	It remains to differentiate the convolution. Fix $R>0$ and choose
	$A>R$. There is a constant $C=C(R,A)$ such that
	\[
		(1+|y|)e^{R|y|}\leq C\bigl(e^{Ay}+e^{-Ay}\bigr)\qquad\text{for every }y\in\R\,.
	\]
	Because $\varrho(y-x)\leq\varrho(y)e^{|x||y|}$ and
	$e^{\pm Ay}\varrho(y)=e^{A^2/2}\varrho(y\mp A)$, uniformly for
	$|x|\leq R$, both $|P(y)|\varrho(y-x)$ and
	$|P(y)(y-x)|\varrho(y-x)$ are bounded by a constant multiple of
	\[
		|P(y)|\bigl(\varrho(y-A)+\varrho(y+A)\bigr)\,.
	\]
	This function is integrable by the shifted integrability proved
	above, applied at $x=A$ and $x=-A$. Differentiation under the
	integral sign is therefore valid locally uniformly in $x$. This
	gives the first formula in \eqref{eq:F-derivative} and proves that
	$F\in C^1(\R)$; the second formula follows by substituting
	$y=x+u$.

	Finally, because $P(Y)=\E[V\mid Y]$ almost surely and
	$Y=\phi(V)+U$, independence of $U$ and $V$ gives
	\[
		\begin{aligned}
			\E[\E[V\mid Y]\mid V]
			 & =\E[P(Y)\mid V]         \\
			 & =\E[P(\phi(V)+U)\mid V] \\
			 & =F(\phi(V))
			\qquad\text{almost surely}\,.
		\end{aligned}
	\]
\end{proof}

\begin{proof}[Proof of Lemma~\ref{lem:foc}]
	By hypothesis, $\phi(v)$ attains the supremum defining $W(v)$; in
	particular, $W$ is real-valued. For every $v,w\in\R$,
	\begin{equation}
		W(w)\geq W(v)+\phi(v)(w-v)\,.
		\label{eq:supporting-line}
	\end{equation}
	For fixed $v$, the function $x\mapsto x(v-F(x))$ is continuously
	differentiable by Lemma~\ref{lem:F-properties} and is maximised at
	$x=\phi(v)$. Its derivative vanishes there, proving
	\eqref{eq:foc}. Substitution into
	$W(v)=\phi(v)(v-F(\phi(v)))$ gives \eqref{eq:W-identity}.

	A finite convex function on $\R$ is locally absolutely continuous
	and differentiable almost everywhere
	\citep[Sections~23--25]{Rockafellar70}. Inequality
	\eqref{eq:supporting-line} states that $\phi(v)$ is a subgradient
	of $W$ at $v$. At every point where $W$ is differentiable, the
	subgradient is unique and equals $W'(v)$. Therefore,
	\eqref{eq:W-prime} holds.

	Let $v_1<v_2$. Applying \eqref{eq:supporting-line} in both
	directions and adding the two inequalities gives
	\[
		\bigl(\phi(v_2)-\phi(v_1)\bigr)(v_2-v_1)\geq 0\,.
	\]
	Thus, $\phi$ is nondecreasing. If $\phi(v_1)=\phi(v_2)$, then
	\eqref{eq:foc} gives $v_1=v_2$, a contradiction. Therefore,
	$\phi$ is strictly increasing.
\end{proof}

\begin{proof}[Proof of Lemma~\ref{lem:posterior}]
	Define, locally in this proof,
	\[
		n(y):=\int_\R v\varrho(v)
		\varrho\bigl(y-\phi(v)\bigr)\,\dd v
		\qquad\text{for every }y\in\R\,,
	\]
	so that $P=n/p_Y$ by the definition of the canonical version of
	the inverse regression. On compact $y$-sets, differentiation under
	the integral sign is justified because $\varrho$ and
	$|t|\varrho(t)$ are bounded, while $|v|\varrho(v)$ is integrable.
	Thus, $p_Y,n\in C^1(\R)$. Because $p_Y$ is everywhere positive by
	Lemma~\ref{lem:densities}, the function $P=n/p_Y$ is continuously
	differentiable.

	The required second moments under $\pi_y$ are locally uniformly
	bounded. Indeed, because $\varrho\leq(2\pi)^{-1/2}$,
	\[
		\int_\R v^2\varrho(v)\varrho\bigl(y-\phi(v)\bigr)\,\dd v
		\leq(2\pi)^{-1/2}\int_\R v^2\varrho(v)\,\dd v
		=(2\pi)^{-1/2}\,.
	\]
	Moreover, if $|y|\leq M$ and $t=y-\phi(v)$, then
	\[
		\phi(v)^2\varrho\bigl(y-\phi(v)\bigr)
		=(y-t)^2\varrho(t)
		\leq2M^2\varrho(t)+2t^2\varrho(t)\,.
	\]
	The right-hand side is bounded uniformly in $t$. Because $p_Y$ is
	continuous and everywhere positive, it is bounded away from zero
	on compact sets. It follows that $\id,\phi\in L^2(\pi_y)$ for every
	$y\in\R$.

	Because
	\[
		\partial_y\varrho\bigl(y-\phi(v)\bigr)
		=\bigl(\phi(v)-y\bigr)
		\varrho\bigl(y-\phi(v)\bigr)\,,
	\]
	we obtain
	\[
		\frac{p_Y'(y)}{p_Y(y)}=\E_{\pi_y}[\phi-y],\qquad
		\frac{n'(y)}{p_Y(y)}
		=\E_{\pi_y}\bigl[\id\,(\phi-y)\bigr]\,.
	\]
	The quotient rule gives
	\[
		P'(y)
		=\E_{\pi_y}\bigl[\id\,(\phi-y)\bigr]
		-\E_{\pi_y}[\id]\,\E_{\pi_y}[\phi-y]
		=\Cov_{\pi_y}(\id,\phi)\,.
	\]
	The last equality holds because adding a constant to the second
	argument leaves the covariance unchanged. Symmetrising the
	covariance gives
	\[
		P'(y)=\frac12\int_\R\int_\R
		(v-w)\bigl(\phi(v)-\phi(w)\bigr)
		\,\pi_y(\dd v)\,\pi_y(\dd w)\,.
	\]
	The density of $\pi_y$ is everywhere positive. Because $\phi$ is
	strictly increasing by Lemma~\ref{lem:foc}, the integrand is
	strictly positive off the diagonal, while the diagonal has zero
	$(\pi_y\otimes\pi_y)$-measure. This proves
	\eqref{eq:P-prime-cov}.

	Let $U_0,U_1$ be independent standard normal random variables.
	By \eqref{eq:F-derivative} of Lemma~\ref{lem:F-properties},
	\[
		F'(x)=\E[U_0P(x+U_0)]\,.
	\]
	The absolute convergence in Lemma~\ref{lem:F-properties} gives
	\[
		U_iP(x+U_i)\in L^1
		\qquad\text{for every }i\in\{0,1\}\,.
	\]
	The same lemma gives $P(x+U_i)\in L^1$. Hence, by independence,
	the cross terms $U_iP(x+U_j)$ are integrable and have expectation
	zero whenever $i\ne j$. Symmetrisation therefore gives
	\[
		2F'(x)
		=\E\bigl[(U_0-U_1)
			\bigl(P(x+U_0)-P(x+U_1)\bigr)\bigr]\,.
	\]
	Because $P$ is strictly increasing, the integrand is nonnegative
	and is strictly positive whenever $U_0\ne U_1$. Because
	$\mathbb P(U_0\ne U_1)=1$, this proves
	\eqref{eq:F-prime-positive}.

	Thus, $P$ and $F$ are strictly increasing.
\end{proof}

\begin{proof}[Proof of Lemma~\ref{lem:resampling}]
	Write $(\Omega,\mathcal F,\mathbb P)$ for the original probability
	space, and let $y\mapsto\pi_y$ be the conditional-law kernel defined above.
	The transfer theorem \citep[Theorem~8.17]{Kallenberg21} guarantees the
	existence of the required conditional replica. In the present setting,
	that consequence can be realised directly from the kernel $\pi$.

	Set
	\[
		\overline\Omega:=\Omega\times\R,
		\qquad
		\overline{\mathcal F}:=\mathcal F\otimes\mathcal B(\R)\,,
	\]
	and define a probability measure on this product space by
	\[
		\overline{\mathbb P}(A)
		:=\int_\Omega\int_\R
		\mathbf 1_A(\omega,v_1)\,\pi_{Y(\omega)}(\dd v_1)
		\,\mathbb P(\dd\omega),
		\qquad A\in\overline{\mathcal F}\,.
	\]
	We identify each original random variable with its lift to
	$\overline\Omega$, set $V^{(0)}:=V$, and define
	\[
		V^{(1)}(\omega,v_1):=v_1\,.
	\]
	By construction, the conditional law of $V^{(1)}$ given the lifted
	sigma-algebra $\mathcal F$ is $\pi_Y$. In particular,
	\[
		\mathcal L(V^{(1)}\mid Y,V^{(0)})=\pi_Y
		\qquad\text{almost surely}\,.
	\]
	The conditional law of $V^{(0)}=V$ given $Y$ is also $\pi_Y$. Hence,
	for bounded Borel functions $f,g\colon\R\to\R$,
	\[
		\begin{aligned}
			\E[f(V^{(0)})g(V^{(1)})\mid Y]
			 & =\E\!\left[
				f(V^{(0)})\int_\R g(v)\,\pi_Y(\dd v)
			\mathrel{\Big|}Y\right]                     \\
			 & =\left(\int_\R f(v)\,\pi_Y(\dd v)\right)
			\left(\int_\R g(v)\,\pi_Y(\dd v)\right)
			\qquad\text{almost surely}\,.
		\end{aligned}
	\]
	Thus, conditionally on $Y$, the variables $V^{(0)}$ and $V^{(1)}$
	are independent and both have law $\pi_Y$. Equivalently,
	\[
		\mathcal L\bigl((V^{(0)},V^{(1)})\mid Y\bigr)
		=\pi_Y\otimes\pi_Y
		\qquad\text{almost surely}\,.
	\]

	The product kernel is invariant under interchange of its coordinates,
	so
	\[
		(V^{(0)},V^{(1)},Y)
		\overset{d}{=}
		(V^{(1)},V^{(0)},Y)\,.
	\]
	This proves~\ref{item:resampling-exchangeable}. Integrating out
	$V^{(0)}$ in the conditional product law gives
	\[
		(V^{(1)},Y)\overset{d}{=}(V,Y)\,,
	\]
	which proves~\ref{item:resampling-marginal}. In particular,
	$V^{(1)}\overset{d}{=}V$, so $V^{(1)}\in L^2$.

	For~\ref{item:resampling-posterior-mean}, the conditional law identified
	above gives
	\[
		\E[V^{(1)}\mid Y,V^{(0)}]
		=\int_\R v\,\pi_Y(\dd v)
		=P(Y)
		\qquad\text{almost surely}\,.
	\]
	Finally, the tower property and~\ref{item:resampling-posterior-mean}
	give
	\[
		\begin{aligned}
			\E[V^{(1)}\mid V^{(0)}]
			 & =\E\!\left[
			\E[V^{(1)}\mid Y,V^{(0)}]
			\mathrel{\Big|}V^{(0)}\right]  \\
			 & =\E[P(Y)\mid V^{(0)}]       \\
			 & =F\bigl(\phi(V^{(0)})\bigr)
			\qquad\text{almost surely}\,.
		\end{aligned}
	\]
	The last equality is the final assertion of
	Lemma~\ref{lem:F-properties}. This proves~\ref{item:resampling-smoothed-mean}.
\end{proof}

\begin{proof}[Proof of Lemma~\ref{lem:lower-gaussian-variance}]
	Let $\varrho$ denote the standard Gaussian density. Because $f$ is
	Lipschitz, it is locally absolutely continuous, its almost-everywhere
	derivative is bounded, and
	\[
		\int_\R |f'(v)|\varrho(v)\,\dd v<\infty\,.
	\]
	For every $R>0$, integration by parts gives
	\[
		\int_{-R}^R f'(v)\varrho(v)\,\dd v
		=
		\bigl[f(v)\varrho(v)\bigr]_{-R}^R
		-
		\int_{-R}^R f(v)\varrho'(v)\,\dd v\,.
	\]
	Because $f$ is bounded and $\varrho(v)\to0$ as $|v|\to\infty$,
	the boundary term tends to zero. Moreover,
	$\varrho'(v)=-v\varrho(v)$, and $v\mapsto |vf(v)|\varrho(v)$ is
	integrable. Dominated convergence therefore gives
	\[
		\E[f'(V)]=\E[Vf(V)]\,.
	\]
	Because $\E[V]=0$,
	\[
		\E[Vf(V)]
		=
		\E\!\left[V\bigl(f(V)-\E[f(V)]\bigr)\right]\,.
	\]
	The Cauchy--Schwarz inequality and $\E[V^2]=1$ now give
	\[
		\bigl|\E[f'(V)]\bigr|
		\leq
		\sqrt{\Var\bigl(f(V)\bigr)}\,.
	\]
	Squaring proves the result.
\end{proof}

\section{From almost-everywhere to pointwise maximisation}
\label{subsec:ae-maximisation}

We return to the setting of Section~\ref{sec:intro}, where
$V\sim\mathcal N(\mu,\tau^2)$ and
$U\sim\mathcal N(0,\sigma^2)$ are independent.
We say that $(\phi,P)$ is a \emph{quasi-equilibrium} if the following conditions
hold:
\begin{enumerate}
	\item (almost everywhere profit maximisation)\label{item:aeprofit-max}
	      for Lebesgue-almost every $v\in\R$,
	      \[
		      \E\bigl[\bigl(v-P(\phi(v)+U)\bigr)\phi(v)\bigr]
		      \geq\E\bigl[\bigl(v-P(x+U)\bigr)x\bigr]\qquad\text{for every }x\in\R\,;
	      \]

	\item (market efficiency)
	      $P(Y_\phi)=\E[V\mid Y_\phi]$ almost surely.
\end{enumerate}

\begin{lemma}
	\label{lem:ae-maximisation}
	If $(\phi,P)$ is a quasi-equilibrium, then
	\begin{equation}
		\phi(v)=\frac{\sigma}{\tau}(v-\mu)
		\qquad\text{for Lebesgue-almost every }v\in\R\,,
	\end{equation}
	and
	\begin{equation}
		P(Y_\phi)=\mu+\lambda Y_\phi
		\qquad\text{almost surely}\,,
	\end{equation}
	where
	\[
		\lambda
		=\frac{\Cov(V,Y_\phi)}{\Var(Y_\phi)}
		=\frac{\tau}{2\sigma}\,.
	\]
\end{lemma}

\begin{proof}[Proof of Lemma~\ref{lem:ae-maximisation}]
	Let $(\phi,P)$ be a quasi-equilibrium. Define
	\[
		\widehat V:=\frac{V-\mu}{\tau},
		\qquad
		\widehat U:=\frac{U}{\sigma},
		\qquad
		\widehat\phi(t):=\frac{\phi(\mu+\tau t)}{\sigma},
		\qquad
		\widehat P(y):=\frac{P(\sigma y)-\mu}{\tau}\,.
	\]
	Then $\widehat V$ and $\widehat U$ are independent standard normal
	random variables, and market efficiency gives
	\[
		\widehat P\bigl(\widehat\phi(\widehat V)+\widehat U\bigr)
		=\E[\widehat V\mid
			\widehat\phi(\widehat V)+\widehat U]
		\qquad\text{almost surely}\,.
	\]
	Define
	\[
		\widehat F(z):=\E[\widehat P(z+\widehat U)],
		\qquad z\in\R\,.
	\]
	The version clause of Lemma~\ref{lem:F-properties} shows that
	$\widehat P(z+\widehat U)\in L^1$ for every $z\in\R$ and that
	$\widehat F$ is continuous. Because
	\[
		P(x+U)
		=\mu+\tau\widehat P\!\left(\frac{x}{\sigma}+\widehat U\right)\,,
	\]
	it follows that $P(x+U)\in L^1$ for every $x\in\R$ and that
	\[
		F(x):=\E[P(x+U)]
		=\mu+\tau\widehat F\!\left(\frac{x}{\sigma}\right)\,,
		\qquad x\in\R\,,
	\]
	is continuous. Consequently, for every $v,x\in\R$,
	\[
		\E\bigl[(v-P(x+U))x\bigr]
		=x\bigl(v-F(x)\bigr)\,.
	\]
	The almost-everywhere profit-maximisation condition now gives
	\[
		\phi(v)\in\argmax_{x\in\R}x\bigl(v-F(x)\bigr)
		\qquad\text{for Lebesgue-almost every }v\in\R\,.
	\]

	Set
	\[
		f(x,v):=x\bigl(v-F(x)\bigr),
		\qquad
		W(v):=\sup_{x\in\R}f(x,v)\,.
	\]
	The function $W$ is convex as the pointwise supremum of affine
	functions of $v$, and $W\geq0$ because $x=0$ is admissible. Let
	$A\subseteq\R$ be a set of full Lebesgue measure on which
	$\phi(v)$ attains the supremum. Then
	\[
		W(v)=f(\phi(v),v)<\infty
		\qquad\text{for every }v\in A\,.
	\]
	Fix $v\in\R$, and choose $a,b\in A$ with $a<v<b$. Convexity gives
	\[
		W(v)
		\leq
		\frac{b-v}{b-a}W(a)
		+
		\frac{v-a}{b-a}W(b)
		<\infty\,.
	\]
	Thus, $W$ is finite on $\R$ and hence continuous.

	For each $v\in\R$, define
	\[
		M(v):=\argmax_{x\in\R}f(x,v)\,.
	\]
	We show that $M(v)$ is nonempty and compact. Choose $a<v<b$. For
	$x\geq0$,
	\[
		f(x,v)
		=f(x,b)-x(b-v)
		\leq W(b)-x(b-v)\,,
	\]
	whereas, for $x\leq0$,
	\[
		f(x,v)
		=f(x,a)+x(v-a)
		\leq W(a)+x(v-a)\,.
	\]
	It follows that $f(x,v)\to-\infty$ as $|x|\to\infty$. Because
	$x\mapsto f(x,v)$ is continuous, it attains its maximum. The same
	limit makes the set of maximisers bounded, while continuity makes it
	closed. Thus, $M(v)$ is nonempty and compact.

	Let $v_1<v_2$ and choose $x_i\in M(v_i)$. Optimality gives
	\[
		f(x_1,v_1)\geq f(x_2,v_1)
		\qquad\text{and}\qquad
		f(x_2,v_2)\geq f(x_1,v_2)\,.
	\]
	Adding these inequalities yields
	\[
		(x_2-x_1)(v_2-v_1)\geq0\,,
	\]
	so $x_1\leq x_2$. Therefore, the function
	\[
		\psi(v):=\min M(v),
		\qquad v\in\R\,,
	\]
	is nondecreasing and hence Borel measurable.

	For every $v\in\R$ and $x\in M(v)$,
	\[
		W(w)
		\geq f(x,w)
		=W(v)+x(w-v)
		\qquad\text{for every }w\in\R\,.
	\]
	Thus, every element of $M(v)$ is a subgradient of $W$ at $v$.
	At every point where $W$ is differentiable, its subgradient is the
	singleton $\{W'(v)\}$; because $M(v)$ is nonempty, this implies
	\[
		M(v)=\{W'(v)\}\,.
	\]
	A finite convex function is differentiable Lebesgue-almost everywhere.
	Intersecting this full-measure differentiability set with $A$ gives
	\[
		\psi(v)=\phi(v)
		\qquad\text{for Lebesgue-almost every }v\in\R\,.
	\]
	By construction,
	\[
		\psi(v)\in\argmax_{x\in\R}x\bigl(v-F(x)\bigr)
		\qquad\text{for every }v\in\R\,.
	\]

	Because the law of $V$ is absolutely continuous and
	$\psi=\phi$ Lebesgue-almost everywhere,
	\[
		\psi(V)=\phi(V)
		\qquad\text{and}\qquad
		Y_\psi=Y_\phi
		\qquad\text{almost surely}\,.
	\]
	Hence market efficiency gives
	\[
		P(Y_\psi)=\E[V\mid Y_\psi]
		\qquad\text{almost surely}\,.
	\]
	Moreover, for every $v,x\in\R$, the payoff representation and the
	pointwise optimality of $\psi(v)$ give
	\[
		\begin{aligned}
			\E\bigl[(v-P(\psi(v)+U))\psi(v)\bigr]
			 & =f(\psi(v),v)                  \\
			 & \geq f(x,v)                    \\
			 & =\E\bigl[(v-P(x+U))x\bigr] \,.
		\end{aligned}
	\]
	Thus, $(\psi,P)$ is an equilibrium. Theorem~\ref{thm:kyle} gives
	\[
		\psi(v)=\frac{\sigma}{\tau}(v-\mu)
		\qquad\text{for every }v\in\R
	\]
	and
	\[
		P(Y_\psi)=\mu+\frac{\tau}{2\sigma}Y_\psi
		\qquad\text{almost surely}\,.
	\]
	Because $\psi=\phi$ Lebesgue-almost everywhere and
	$Y_\psi=Y_\phi$ almost surely, these are exactly the two conclusions
	of the lemma. The covariance formula for $\lambda$ follows from the
	corresponding formula in Theorem~\ref{thm:kyle} and the almost-sure
	equality of $Y_\psi$ and $Y_\phi$.
\end{proof}

\makeatletter
\NAT@openbibfalse
\makeatother
\bibliographystyle{plainnat}
\bibliography{kyle}
\end{document}